\documentclass[12pt]{amsart}
\usepackage{amscd}
\usepackage{amsmath}
\usepackage{amsthm}
\usepackage{amsfonts}
\usepackage{amssymb}
\usepackage{color}

\numberwithin{equation}{section}
\newtheorem{pro}{Proposition}[section]
\newtheorem{prp}[pro]{Proposition}
\newtheorem{lmm}[pro]{Lemma}
\newtheorem{thm}[pro]{Theorem}
\newtheorem{rem}[pro]{Remark}
\newtheorem{crl}[pro]{Corollary}
\newtheorem*{corollary*}{Corollary}

\newtheorem*{thm*}{Theorem}

{\theoremstyle{definition}
\newtheorem{dfn}[pro]{Definition}

 }

\newcommand{\al}{\alpha}
\newcommand{\be}{\beta}
\newcommand{\ga}{\gamma}

\newcommand{\lam}{\lambda}

\newcommand{\ra}{\rightarrow}
\newcommand{\lra}{\longrightarrow}

\newcommand{\CC}{{\mathbb C}}

\newcommand{\RR}{{\mathbb R}}

\newcommand{\ZZ}{{\mathbb Z}}

\newcommand{\Aut}{\mathop{\rm Aut}\nolimits}

\newcommand{\Ker}{\mathop{\rm Ker}}

\newcommand{\Rank}{\mathop{\rm Rank}\nolimits}

\newcommand{\N}{\mathalpha{\mathrm N}}  
\newcommand{\til}{\tilde}

\newcommand{\cK}{{\mathcal{K}}}

\newcommand{\cN}{{\mathcal{N}}}

\newcommand{\cB}{{\mathcal{B}}}
\newcommand{\cG}{{\mathcal{G}}}

\newcommand{\cf}{{cf$.$\,}}

\newcommand{\ie}{{i.e$.$\,}}

\begin{document}

\pagestyle{myheadings} \thispagestyle{empty} \setcounter{page}{1}

\title[]
{On the $S^1$-fibred nil-Bott Tower }

\author[]{Mayumi Nakayama}
\address{Department of Mathematics and
Information of Sciences\\
Tokyo Metropolitan
University\\
Minami-Ohsawa 1-1,
Hachioji, Tokyo 192-0397, JAPAN}
\email{nakayama-mayumi@ed.tmu.ac.jp}

\date{\today}
\keywords{Aspherical manifolds, Bott tower,
Seifert fiber spaces, Infranil manifolds}
\subjclass[2000]{53C55, 57S25, 51M10\ (this file
draft niltower)}

\maketitle

\begin{abstract}
We shall introduce a notion of \emph{$S^1$-fibred nilBott tower}.
It is an iterated $S^1$-bundles whose top space is called an $S^1$-fibred nilBott manifold and the $S^1$-bundle of each stage realizes a \emph{Seifert construction}. The nilBott tower is a generalization of \emph{real Bott tower}
from the viewpoint of fibration.
In this note we shall prove that any $S^1$-fibred nilBott manifold
is \emph{diffeomorphic} to an infranilmanifold.
According to the group extension of each stage,
there are two classes of  $S^1$-fibred nilBott manifolds which
is defined as  \emph{finite type} or \emph{infinite type}.
We discuss their properties.
\end{abstract}

\section{Introduction}
Let $M$ be a closed aspherical
manifold which is a top space of an iterated $S^1$-bundles over a point: 
\begin{equation}\label{tower}
 M=M_{n}\ra M_{n-1}\ra\dots \ra M_{1}\ra \{{\rm pt}\}.
\end{equation}
Suppose $X$ is the universal covering of $M$ and each
$X_i$ is the universal covering of $M_i$ and put $\pi_1(M_i)=\pi_i \
(i=1,\dots, n-1)$ and $\pi_1(M)=\pi$.

\begin{dfn}\label{def1}
An \emph{$S^1$-fibred nilBott tower} is a sequence
\eqref{tower} which satisfies  I, II and III below $(i=1,\dots,n-1)$.
The top space $M$ is said to be an \emph{$S^1$-fibred nilBott manifold
(of depth $n$).}
\begin{itemize}
\item[I.] $M_i$ is
a  fiber space over $M_{i-1}$ with fiber $S^1$.
\item[II.] For the group extension
\begin{equation}\label{start1}
 1\ra \ZZ\ra \pi_i\lra\pi_{i-1}\ra 1
\end{equation}associated to the fiber space (I), there is an equivariant principal bundle:
\begin{equation}\label{nilprincipal}
\RR\ra X_i\stackrel{p_i}\lra X_{i-1}.
\end{equation}
\item[III.] Each $\pi_i$ normalizes $\RR$.
\end{itemize}
\end{dfn}

The purpose of this paper is to prove the following result.

\begin{thm}\label{s1-case}
Suppose that $M$ is an $S^1$-fibred nilBott manifold.

\begin{itemize}
\item[(I)] If every cocycle of $H^2_\phi(\pi_{i-1};\ZZ)$ which
represents a group extension \eqref{start1} is \emph{of finite order},
then $M$ is diffeomorphic to a Riemannian flat manifold.

\item[(II)] If there exists a cocycle of $H^2_\phi(\pi_{i-1};\ZZ)$
 which represents a group extension \eqref{start1} is
\emph{of infinite order}, 
  then $M$ is diffeomorphic to an infranilmanifold. In addition, $M$
  cannot be diffeomorphic to any Riemannian flat manifold.
\end{itemize}
\end{thm}

\tableofcontents

\section{Seifert construction}
We shall explain the Seifert construction brieffly. It is a tool to construct a closed asphrical manifold for a given extension. \\

Let
\begin{equation}\label{Q-extension}
\begin{CD}
1@>>> \Delta@>>> \pi@>{\nu}>> Q@>>>1
\end{CD}
\end{equation}
be a group extension. Then there is a conjugation function $\phi:Q\ra {\Aut}(\Delta) $ defined by
a section $s:Q\ra \pi$ of $\nu$. The group extension
\eqref{Q-extension} is represented by a cocycle $f:Q\times Q\ra \Delta$ for which each element $\ga\in \pi$ is viewed as $(n,\al)$
with group law: \[
(n,\al)(m,\be)=(n\cdot \phi(\al)(m)\cdot f(\al,\be),\al\be) \]  $(\forall\,
n,m\in\Delta,\,\forall\, \al,\be\in Q)$ (\cf \cite{Mayu} for example).\\

 Suppose $\Delta$ is a torsionfree
 finitely generated nilpotent group. By Mal'cev's {\em existence} theorem, there is a (simply
connected) nilpotent Lie group $\cN$ containing $\Delta$ as a
discrete uniform subgroup. Moreover if $Q$ acts properly discontinuously on a contractible smooth manifold $W$ 
such that the quotient space $W/Q$ is compact, then there is a smooth map $\lam:Q\ra {\rm Map}(W,\mathcal N)$ satisfies $f=\delta^1 \lam$:
\begin{equation}\label{co}
f(\al,\be)=(\bar{\phi}(\al)\circ \lam(\be)\circ \al^{-1})\cdot \lam(\al)\cdot \lam(\al\be)^{-1}\ \ (\al,\be\in Q)
\end{equation} 
here  $\bar \phi:Q\ra \cN$ is the extension of $\phi$. And an action of $\pi$ on $\cN\times W$ 
is obtain by 
\begin{equation}\label{sei-ac}
(n,\al)(x,w)=(n\cdot \bar\phi(\al)(x)\cdot \lam(\al)(\al w),\al w).
\end{equation}
This action $(\pi,\cN\times W)$ is said to be a Seifert
construction. (See \cite{KLR} for details.)\\

Taking a finite group $F$ and $\{ pt \}$ as $Q$ and $W$ above; 
\begin{equation}\label{ext1} 
\begin{CD}
1@>>> \Delta@>>> \pi@>{\nu}>> F@>>>1,
\end{CD}
\end{equation}
we may put  $\cN$ as  ${\rm Map}(W,\mathcal N)$ before respectively. Let $\cK$ be a maximal compact subgroup of $\Aut(\cN)$.  The group  ${\rm E}(\cN)=\cN\rtimes \cK $ is said to be  the euclidean group of
$\cN$ . Then there is a  discrete faithful representaition $\rho:\pi\ra {\rm E}(\cN)$  which is defined by  
\begin{equation}\label{faith}
\rho((n,\al))=(n\cdot \chi(\al), \, \mu(\chi(\al)^{-1})\circ\bar\phi(\al)) \  (n\in \Delta, \al\in F),
\end{equation}
where $\chi:F\ra \cN$ is a map such that $f=\delta^1 \chi$:
\begin{equation}\label{cocycle}
f(\al,\be)=\bar{\phi}(\al)(\chi(\be))\cdot \chi(\al)\cdot \chi(\al\be)^{-1}\ \ (\al,\be\in Q).
\end{equation}
(See \cite{Mayu}.) Note that the action $(\rho(\pi), \cN)$ is a Seifert construction and  $\mathcal N/\rho(\pi)$ is an infranilmanifold 
(\cf \cite{KLR} or \cite{Mayu}).

\section{$S^1$-fibred nilBott tower}\label{nilBottS}
 This section gives the proof Theorem \ref{s1-case}.\\
 
Suppose that 
\begin{equation}\label{towerS} \begin{CD}
 M=M_n@>S^1>> M_{n-1}@>S^1>>\dots
@>S^1>> M_{1}@>S^1>> \{{\rm pt}\}\end{CD}
 \end{equation} is an $S^1$-fibred nilBott tower.
By the definition, there is a group extension of the fiber space;
\begin{equation}\label{start2} 1\ra \ZZ\ra
\pi_i\lra \pi_{i-1}\ra 1
\end{equation}for any $i$. 
The conjugate by each element of $\pi_i$ defines a
homomorphism $\phi:\pi_{i-1}\ra {\rm Aut}(\ZZ)=\{\pm 1\}$. With this action,
$\ZZ$ is a $\pi_{i-1}$-module so that the group cohomology $H^i_\phi(\pi_{i-1},\ZZ)$
is defined. Then the above group extension \eqref{start2} represents a $2$-cocycle
in $H^2_\phi(\pi_{i-1},\ZZ)$,  (\cf \cite{Mayu}).
\begin{proof}
Given a group extension \eqref{start2}, we suppose by induction that
there exists a torsionfree finitely generated
nilpotent normal subgroup $\Delta_{i-1}$ of finite index in $\pi_{i-1}$
such that the induced extension $\til{\Delta_i}$  is a central extension:
\begin{equation}\label{nilgroupex}
\begin{CD}
1@>>> \ZZ@>>> \pi_{i}@>>>\pi_{i-1}@>>> 1\\
@. ||@. @AAA @AAA \\
1@>>> \ZZ@>>> \til{\Delta_{i}}@>>>
\Delta_{i-1}@>>> 1.\\
\end{CD}\end{equation}
It is easy to see that $\til{\Delta_i}$ is a torsionfree finitely generated normal nilpotent
subgroup of finite index in $\pi_i$. Then $\pi_i$ is a virtually nilpotent subgroup, \ie
$\displaystyle 1\ra \til{\Delta_i}\ra \pi_i\lra F_i\ra 1$ where
$F_i=\pi_i/\til{\Delta_i}$ is a finite group. Let $\til{N_i}$, $N_{i-1}$be a nilpotent Lie
group containing $\til{\Delta_i}$, $\Delta_{i-1}$ as a discrete cocompact subgroup respectively. 
Let ${\rm A}(\til{N_i})=\til{N_i}\rtimes {\rm Aut}(\til{N_i})$ be the affine group. If
$\til{K_i}$ is a maximal compact subgroup of ${\rm Aut}(\til{N_i})$, then the
subgroup ${\rm E}(\til{N_i})=\til{N_i}\rtimes \til{K_i}$ is the euclidean group of
$\til{N_i}$. Then there exists a faithful homomorphism (see \eqref{faith}):
\begin{equation}\label{infarnil}
\rho_i: \pi_i\lra {\rm E}(\til{N_i})
\end{equation}for which $\rho_i|_{\til{\Delta_i}}={\rm id}$ and the quotient $\til{N_i}/\rho_i(\pi_i)$ is
an infranilmanifold. 
The explicit formula is given by the following 
\begin{equation}\label{nil-rep}
\rho_i((n,\al))=(n\cdot \chi(\al),\mu(\chi(\al)^{-1})
\circ\bar\phi(\al))\
\end{equation}for $n\in \til \Delta_i, \al\in F$  where $\chi:F\ra \til \Delta_i, \ \bar \phi:F\ra {\rm Aut}(\til \N_i).$ 
As $\til{\Delta_i}\leq \til{N_i}$, there is a
$1$-dimensional vector space $\mathsf R$ containing $\ZZ$ as a
discrete uniform subgroup which has a central group extension:
\[ 1\ra \mathsf R\ra \til{N_i}\lra N_{i-1}\ra 1\] where  $N_{i-1}=\til{N_i}/\mathsf R$ is a simply
connected nilpotent Lie group. As $\ZZ\leq \mathsf R\cap \til{\Delta_i}$
is discrete cocompact in $\mathsf R$ and $\displaystyle \mathsf R\cap
\til{\Delta_i}/\ZZ\ra \til{\Delta_i}/\ZZ\cong \Delta_{i-1}$ is an inclusion,
noting that $\Delta_{i-1}$ is torsionfree, it follows that $\mathsf
R\cap \til{\Delta_i}=\ZZ$. We obtain the commutative diagram in which the
vertical maps are inclusions:
\begin{equation}\label{inductionnil}
\begin{CD}
1@>>> \ZZ@>>> \til{\Delta_{i}}@>>>\Delta_{i-1}@>>> 1\\
@. @VVV @VVV @VVV \\
1@>>> \mathsf R@>>> \til{N_i}@>>>
N_{i-1}@>>> 1.\\
\end{CD}
\end{equation}

On the other hand, \eqref{infarnil} induces the following group
extension.:
\begin{equation}\label{inductionnil2}
\begin{CD}
1@>>> \ZZ@>>> \pi_{i}@>p_i>>\pi_{i-1}@>>> 1\\
@. @| @V{\rho_i}VV @V{\hat\rho_i}VV \\
1@>>> \ZZ@>>>\rho_i(\pi_i)@>>>
\hat\rho_i(\pi_{i-1})@>>> 1.\\
\end{CD}
\end{equation}
Since $\til{\Delta_i}$ centralizes $\ZZ$, $\til{N_i}$ centralizes
$\mathsf R$. So $\hat \rho_i$ is a homomorphsim
from $\pi_{i-1}$ into ${\rm E}(N_{i-1})$. 
The explicit formula is given by the following:
\begin{equation}\label{nil-1-rep}
\hat\rho_i((\bar n,\al))=(\bar n\cdot \bar \chi(\al),\mu(\bar \chi(\al)^{-1})
\circ\hat\phi(\al))
\end{equation}for $\bar n\in \Delta_{i-1}, \, \al\in F$ where $\bar \chi=p_i\circ \chi:F\ra \Delta_{i-1}, \ \hat\phi:F\ra {\rm Aut}(\N_{i-1});$
\begin{equation*}\label{nil-2-rep}
\begin{split}
\hat\phi(\al)(\bar x)&=\overline{\bar\phi(\al)(x)}.
\end{split}
\end{equation*}Note that $\bar\phi(\al)(\mathsf R)=\mathsf R$. 
As the action $(\pi_{i-1},N_{i-1})$
is properly discontinuous and $\pi_{i-1}$ is torsionfree,
the representation $\displaystyle \hat\rho_{i}:
\pi_{i-1}\ra {\rm E}(N_{i-1})$ is faithful.
(Note that $\hat \rho_{i}|_{\Delta_{i-1}}={\rm id}$.)
Thus we obtain an equivariant fibration:
\begin{equation}\label{inductionnil3}
\begin{CD}
(\ZZ,\mathsf R)@>>>(\rho_i(\pi_i),\til{N_i})@>\nu_i>>(\hat\rho_{i}(\pi_{i-1}),N_{i-1}).
\end{CD}
\end{equation}

Suppose by induction that $(\pi_{i-1},X_{i-1})$ is equivariantly
diffeomorphic to the infranil-action
$(\hat\rho_{i}(\pi_{i-1}),N_{i-1})$ as above. We have two Seifert
fibrations from \eqref{nilprincipal} where
\begin{equation*}
(\ZZ,\mathsf R)\ra (\pi_i, X_i)\stackrel{p_i}\lra
(\pi_{i-1},X_{i-1})
\end{equation*}and \eqref{inductionnil3} where
\begin{equation*}
(\ZZ,\mathsf
R)\ra(\rho_i(\pi_i),\til{N}_i)\stackrel{\nu_i}\lra(\hat\rho_{i}(\pi_{i-1}),N_{i-1}).
\end{equation*}
As $\rho_i:\pi_i\ra\rho_i(\pi_i)$ is isomorphic such that
$\rho_i|_{\ZZ}={\rm id}$, the Seifert rigidity implies that $(\pi_i,
X_i)$ is equivariantly diffeomorphic to $(\rho_i(\pi_i),\til{N_i})$. This
shows the induction step.
Let $M=X/\pi$. Then $(\pi,X)$ is equivariantly diffeomorphic to an
infra-nilaction $(\rho(\pi),\til{N})$ for which $\rho:\pi\ra E(\til{N})$ is a
faithful representation.\\

We have shown that $M$ is diffeomorphic to an infranilmanifold
$\til{N}/\rho(\pi)$.  According to Cases I, II (stated in Theorem
\ref{s1-case}), we prove that $\til{N}$ is isomorphic to a vector space
for Case I or $\til{N}$ is a nilpotent Lie group but not a vector space
for
Case II \,  respectively.\\
  \ Case I. \,  As every cocycle of $H^2_\phi(\pi_{i-1},\ZZ)$ representing
 a group extension \eqref{start2} is finite, the cocycle in
$H^2(\Delta_{i-1},\ZZ)$  for  the induced extension of
\eqref{nilgroupex} that $\displaystyle 1\ra \ZZ\ra \til{\Delta_{i}}\lra
\Delta_{i-1}\ra 1$ is also finite. By induction, suppose that
$\Delta_{i-1}$ is isomorphic to a free abelian group $\ZZ^{i-1}$.
Then the cocycle in $H^2(\ZZ^{i-1},\ZZ)$
is zero, so $\til{\Delta_{i}}$ is isomorphic to a free abelian group
$\ZZ^{i}$. Hence the nilpotent Lie group $N_i$ is
isomorphic to the vector space $\RR^i$. This shows the induction
step. In particular, $\pi_i$ is isomorphic to a Bieberbach group
$\rho_i(\pi_i)\leq {\rm E}(\RR^i)$. As a consequence $X/\pi$ is
diffeomorphic to a Riemannian flat manifold $\RR^n/\rho(\pi)$.\\
  \ Case II. Suppose that $\pi_{i-1}$ is virtually free abelian until
$i-1$ and the cocycle $[f]\in H^2_{\phi}(\pi_{i-1},\ZZ)$
representing a group extension $\displaystyle 1\ra \ZZ\ra \pi_i\lra
\pi_{i-1}\ra 1$ is of infinite order in $H^2_\phi(\pi_{i-1},\ZZ)$.
Note that $\pi_{i-1}$ contains a torsionfree normal free abelian
subgroup $\ZZ^{i-1}$. As in \eqref{nilgroupex}, there is a central
group extension of $\til{\Delta_i}$:
\begin{equation}\label{infinitenil}
\begin{CD}
1@>>> \ZZ@>>> \pi_{i}@>>>\pi_{i-1}@>>> 1\\
@. ||@. @AAA @AA{\rm i}A \\
1@>>> \ZZ@>>> \til{\Delta_{i}}@>>>
\ZZ^{i-1}@>>> 1\\
\end{CD}\end{equation}where $[\pi_{i-1}:\ZZ^{i-1}]<\infty$.
Recall that 
there is a transfer homomorphism $\tau:
H^2(\ZZ^{i-1},\ZZ)\ra H^2_\phi(\pi_{i-1},\ZZ)$  such that
$\tau\circ{\rm i}^*=[\pi_{i-1}:\ZZ^{i-1}]:
H^2_\phi(\pi_{i-1},\ZZ)\ra H^2_\phi(\pi_{i-1},\ZZ)$, see
\cite[(9.5) Proposition p.82]{BRO} for example. 
The restriction ${\rm
i}^*[f]$ gives the bottom  extension sequence of
\eqref{infinitenil}. If ${\rm i}^*[f]=0\in H^2(\ZZ^2,\ZZ)$, then $0=\tau\circ{\rm i}^*[f]=[\pi_{i-1}:\ZZ^{i-1}][f]\in H^2_\phi(\pi_{i-1},\ZZ)$. So ${\rm i}^*[f]\not=0$. 
Therefore $\til{\Delta_{i}}$ (respectively $\til{N_i}$) is not abelian
(respectively not isomorphic to a vector space). As a consequence,
$\til N$ is a simply connected (non-abelian) nilpotent Lie group.
\end{proof}
 In order to study $S^1$-fibred nilBott manifolds further,
we introduce the following definition:
\begin{dfn}\label{finite/infinite}
If an $S^1$-fibred nilBott manifold
$M$ satisfies Case I (respectively Case II)  of Theorem \ref{s1-case},
then $M$ is said to be an $S^1$-fibred nilBott manifold of finite type
(respectively of infinite type).
Apparently there is no intersection between  finite type
and infinite type. And $S^1$-fibred nilBott manifolds are  of finite type
until dimension $2$. 
\end{dfn}

\begin{rem}Let $M$ be an $S^1$-fibred nilBott manifold of finite type, then $\rho(\pi)$ is a Bieberbach group (\cf Theorem \ref{s1-case}). 
By the Bieberbach Theorem, $\rho(\pi)$ satisfies a group extension 
\begin{equation}\label{start3}
1\ra \ZZ^{n}\ra \rho(\pi)\lra H\ra 1
\end{equation}
where $\ZZ^{n}=\rho(\pi)\cap \RR^{n}$, and $H$ is the holonomy group 
of $\rho(\pi)$. 
We may identify $\rho(\pi)$ with $\pi$
whenever $\pi$ is torsionfree. 
\end{rem}
\noindent 
\begin{prp}\label{hol-Z2}
Suppose $M$ is an $S^1$-fibred nilBott manifold of finite type.
Then the holonomy group of $\pi$ 
is isomorphic to the power of cyclic group of order
two $(\ZZ_2)^s$ in $(0\leq s\leq n)$.
\end{prp}

\begin{proof}  Let $M$ be an $S^1$-fibred nilBott manifold of finite type. From \eqref{start2} recall a group
extension
\begin{equation}\label{start3}
 1\ra \ZZ\ra \pi_i\stackrel{p_i}\lra \pi_{i-1}\ra 1
 \end{equation}
  which associates to the equivariant fibration:
\begin{equation*}
(\ZZ,\mathsf R)\ra (\pi_i, \til{N_i})\stackrel{p_i}\lra
(\pi_{i-1},N_{i-1}).
\end{equation*}
If $f$ is a cocycle in $H^2_\phi(\pi_{i-1},\ZZ)$ for Case I
representing \eqref{start3},
 then there exists a map $\lam:\pi_{i-1}\ra \RR$ such that
\begin{equation}\label{kcycle}
f(\al,\be)=\bar\phi(\al)(\lam(\be))+\lam(\al)-\lam(\al\be) \ \,
(\al,\be\in \pi_{i-1})
\end{equation} (see \cite{C-R}).
Moreover let $(n,\al)\in \pi_i$ and $(x,w)\in \til{N_i}=\mathsf R\times
N_{i-1}$, then the action of $\pi_i$ is given by
\begin{equation}\label{Seifertaction1}
\begin{CD}
(n,\al)(x,w)=(n+\bar\lam(\al)(x)+\lam(\al),\al w)\\
\end{CD}
\end{equation}
(Remark that $n\in\ZZ, \, \al\in\pi_{i-1}$ and see \eqref{sei-ac}.) As we have shown in Case I of Theorem
\ref{s1-case}, $(\pi_i,\til{N_i})$ is a Bieberbach group action. Let
$(\pi_{i-1},N_{i-1})=(\pi_{i-1},\RR^{i-1})$ where $\pi_{i-1}\leq
{\rm E}({i-1})=\RR^{i-1}\rtimes {\rm O}(i-1)$ is a Bieberbach group
such that
\begin{equation*}
\al w=b_\al+A_\al w \ (w\in \RR^{m-1})
\end{equation*} here $b_\al\in \RR^m, \, A_\al\in {\rm O}(i-1)$ in the above action of \eqref{Seifertaction1}.

Let $L:{\rm E}({i-1})\ra {\rm O}(i-1)$ be the linear holonomy homomorphism. Suppose inductively that $L(\pi_{i-1})=\{A_\al\, |\, \al\in
\pi_{i-1}\} \leq (\ZZ_2)^{i-1}$. Here
\begin{equation}\label{2torsiongroup}
(\ZZ_2)^{i-1}=\{\left(\begin{array}{ccc}
\pm 1&       &    \\
     &\ddots &     \\
     &       &\pm 1 \\
\end{array}\right)\}\leq {\rm O}(i-1). 
\end{equation}
 
Then the above action \eqref{Seifertaction1} has the formua:
\begin{equation}\label{Seifertaction2}
 (n,\al)\left[\begin{array}{c}
 x\\
 w \end{array}\right]=\Biggl(\left(\begin{array}{c}
  n+\lam(\al)\\
b_\al\end{array}\right),\left(\begin{array}{cc} \bar\phi(\al) &  \mbox{{\large $0$}} \\
\mbox{{\large $0$}} &
A_\al\end{array}\right)\Biggr)\left[\begin{array}{c}
 x\\
 w \end{array}\right],
\end{equation} where $\displaystyle\left[\begin{array}{c}
 x\\
 w \end{array}\right]\in \til{N_i}=\mathsf R\times \RR^{i-1}=\RR^i$. It follows $(n,\al)\in {\rm E}(i)$. 
Since $\bar\phi:\pi_{i-1}\ra
\{\pm 1\}\leq \Aut(\RR)$ is a unique extension of $\phi:\pi_{i-1}\ra
\Aut(\ZZ)=\{\pm 1\}$, we see that the 
$H_i$ isomorphism $(\ZZ_2)^s, (\ 0\leq s\leq i$). This proves the induction step. 
\end{proof}

\begin{crl}\label{torus-action}
Each $S^1$-fibred nilBott manifold of finite type $M_i$ 
admits a homologically injective $T^k$-action where
$k={\rm Rank}\, H_1(M_i)$. Moreover, the action is maximal,\ie
$k={\rm Rank}\, C(\pi_i)$.
\end{crl}

\begin{proof}
We suppose by induction that there is a
\emph{homologically injective} maximal $T^{k-1}$-action on
$M_{i-1}=T^{i-1}/H$ such that $k-1={\rm Rank}\,H_1(M_{i-1})={\rm
Rank}\,C(\pi_{i-1})$  $(k-1> 0)$. Since $\pi_i$, $\pi_{i-1}$ are Bieberbach groups, there are two  group extensions 
\[
1\ra \ZZ^i\ra \pi_i\stackrel{h_i}\lra H_i\ra 1\] 
\[
1\ra \ZZ^{i-1}\ra \pi_{i-1}\stackrel{h_{i-1}}\lra {H_{i-1}}\ra 1\] 
where $H_i$, $H_{i-1}$ are  holonomy groups of $\pi_i$, $\pi_{i-1}$, respectively and $\ZZ^i=\pi_i\cap \RR^i$, $\ZZ^{i-1}=\pi_{i-1}\cap \RR^{i-1}$. We have a following diagram
\begin{equation}\label{inductionnil}
\begin{CD}
@. @. 1@. 1@.\\
@. @. @VVV @VVV @. \\
1@>>> \ZZ@>>> \ZZ^i@>>>\ZZ^{i-1}@>>> 1\\
@. @| @VVV @VVV \\
1@>>> \ZZ@>>> \pi_i@>{p_i}>>\pi_{i-1}@>>> 1\\
@. @. @VV{h_i}V @VV{h_{i-1}}V @.\\
@. @. H_i @= H_{i-1} @.\\
@. @. @VVV @VVV @.\\
@. @. 1@. 1@. @.\\
\end{CD}
\end{equation}
Let
$\displaystyle p:\RR^i=\mathsf R\times \RR^{i-1}\ra T^i=S^1\times
T^{i-1}$ be the canonical projection such that $\Ker\, p=\ZZ^i=\pi_i\cap \RR^i$.
By Proposition \ref{hol-Z2}, $H_i=(\ZZ_2)^s$ for some $s$
($1\leq s \leq i$). The action $(\pi_i,\RR^i)$ induces an
isometric action $(H_i,T^i)$ from \eqref{Seifertaction2}.  We may represent the action  as the following
\begin{equation}\label{Seiferttorus3}
 \hat\al
 \left(\begin{array}{c}
 z_1\\
 z_2\\
\vdots\\
z_i
\end{array}\right)=\left(\begin{array}{c}
 t_{\hat\al}\cdot \psi(\hat\al)(z_1)\\
 {z_2}'\\
\vdots\\
 {z_i}'
\end{array}\right)
\end{equation}
here $\hat\al=h_i((n,\al))\in H_i$, 
$t_{\hat\al}=p(n+\lam(\al))\in S^1$, 
and $\psi:H_i\ra \{{\pm 1}\}$, 
\begin{equation}  
\psi(\hat\al)(z_1)=
\left\{\begin{array}{cc} z_1 &\, \mbox{if}\ \bar\phi(\al)=1 \\
\bar z_1 &\ \ \, \mbox{if}\ \bar\phi(\al)=-1. \\
\end{array}\right.
\end{equation}
Note that $(t_{\hat\al})^2=p(n+\lam(\al))p(n+\lam(\al))=p(2n+2\lam(\al)).$ 
Suppose $\bar\phi(\al)=1$. By \eqref {Seifertaction2}, 
\begin{equation}\label{Seifertaction}
 (n,\al)^2\left[\begin{array}{c}
 x\\
 w \end{array}\right]=\Biggl(\left(\begin{array}{c}
  2n+2\lam(\al)\\
b_\al+A_\al w\end{array}\right),\left(\begin{array}{ccc}
1&          \\
 &\ddots &  \\
 &       &1 \\
\end{array}\right) 
\Biggr)\left[\begin{array}{c}
 x\\
 w \end{array}\right].
\end{equation}
Since $2n+2\lam(\al)\in \ZZ$, then $(t_{\hat\al})^2=1$ \ie $t_{\hat\al}=\pm 1$.\\
 If $H_i=\Ker\,\psi,$ it follows from \eqref{Seiferttorus3} that the
$S^1$-action on $T^i=S^1\times T^{i-1}$ as left translations induces an $S^1$-action on
$M_i=T^i/H_i$ so that $T^{k}$-action on
$M_i=T^i/H_i$ follows 
\begin{equation}\label{Seiferttorus3}
 \left(\begin{array}{c}
 t\\
 t'\\
\end{array}\right)
 \left[\begin{array}{c}
 z_1\\
 z_2\\
\vdots\\
z_i
\end{array}\right]=\left[\begin{array}{c}
 t \cdot z_1\\
 t' \cdot {\left(\begin{array}{c}{z_2}\\
\vdots\\
 {z_i}
\end{array}\right)}
\end{array}\right]
\end{equation} 
where $(t,t')\in {S^1\times T^{k-1}}, [z_1,\dots,z_i]\in {M_i=T^i/H_i}$. 
On the other hand, if there is an element $\hat\al$ of $H_i$ which $\psi(\hat\al)(z)=\bar z,$ then $M_i$ admits a $T^{k-1}$-action by the inducton hypothesis. The group extension \eqref{start3} gives rise to a group extension:
\begin{equation}\label{start4}
 1\ra \ZZ/[\pi_i,\pi_i]\cap \ZZ\ra \pi_i/[\pi_i,\pi_i]\stackrel{\nu_i}
 \lra \pi_{i-1}/[\pi_{i-1},\pi_{i-1}]\ra 1.
 \end{equation}
As in the proof of Proposition \ref{hol-Z2}, $[(0,\al),(n,1)]=
((\phi(\al)-1)(n),1)$. It follows that $[\pi_i,\pi_i]\cap \ZZ=\{1\}$
or $[\pi_i,\pi_i]\cap \ZZ=2\ZZ$ according to whether $H_i=\Ker\,\psi$
or not. So \eqref{start4} becomes
\begin{equation}\label{start5}
 \begin{CD}
 1@>>>\ZZ@>>>H_1(M_i)@>{\nu_i}>>H_1(M_{i-1})@>>>1,
 \end{CD}
\end{equation}
or
\begin{equation}\label{start6}
 \begin{CD}
 1@>>>\ZZ_2@>>>H_1(M_i)@>{\nu_i}>>H_1(M_{i-1})@>>>1.
\end{CD}
\end{equation}
For \eqref{start5}, it follows $k={\rm Rank}\,H_1(M_{i})$ for
which $M_i$ admits a homologically injective $T^{k}$-action as
above. For \eqref{start6}, $k-1={\rm Rank}\,H_1(M_{i})$ and $M_i$
admits a homologically injective $T^{k-1}$-action by the induction
hypothesis.\\
 \ Suppose $H_i=\Ker\,\psi$. Noting that the group extension $1\ra \ZZ\ra \pi_i\stackrel{p_i}\lra \pi_{i-1}\ra 1$ is a central extension,  
we obtain a group extension:
\[
1\ra \ZZ\ra C(\pi_i)\stackrel{p_i}\lra p_i(C(\pi_{i}))\ra 1.
\]
On the other hand, since $M_i$ admits the above $T^k$-action,  $\ZZ^k\subset C(\pi_i).$ Let $\Rank\,C(\pi_i)=k+l,$ $(l=0,1,2,\dots)$, then $\ZZ^{k+l-1}\subset p_{i}(C(\pi_i))$. By the induction hypothesis, $k-1=\Rank\,C(\pi_{i-1})\geq \Rank\,p_i(C(\pi_{i}))$. 
Therefore as $l=0$, $k=\Rank\,C(\pi_i)$.\\
 \ Assume that there exists an element $\hat\al\in H_i$
 such that $\psi(\hat\al)(z)=\bar z$. It easy to check that $\ZZ\cap C(\pi_i)=\phi$, i.e. $C(\pi_{i})\leq C(\pi_{i-1})$ and since $M_i$ admits $T^{k-1}$-action, $\ZZ^{k-1}\leq C(\pi_i)$. By the induction hypothesis, $k-1={\rm Rank}\,C(\pi_i).$\\
Therefore in each case the torus action is maximal.
\end{proof}

\section{3-dimensional $S^1$-fibred nilBott towers} By the definition of 
$S^1$-fibred nilBott manifold $M_n$,
$M_2$ is either a torus $T^2$ or a Klein bottle $K$ so that $M_2$ is a Riemannian flat manifofld. 

\subsection{3-dimensional $S^1$-fibred nilBott manifolds of finite type}
Any $3$-dimensional $S^1$-fibred nilBott manifold $M_3$ of finite type is a Riemannian flat manifofld. It is known that there are just 
$10$-isomorphism classes $\cG_1,\dots,\cG_6$,
$\cB_1,\dots,\cB_4$ of $3$-dimensional Riemannian flat
manifolds. (Refer to the classification of $3$-dimensional
Riemannian flat manifolds  by Wolf \cite{WO}.) In particular, for Riemannian flat $3$-manifolds corresponding to
$\cB_2$ and $\cB_4$, 
we have shown that there are two $S^1$-fibred nilBott towers:
$\cB_2\ra K\ra S^1\ra  \{{\rm pt}\}$ and $\cB_4\ra K\ra S^1\ra  \{{\rm pt}\}$ in \cite{Mayu}. Remark that every real Bott manifold is an
$S^1$-fibred nilBott manifold of finite type and ${\cB}_2$ and ${\cB}_4$ are not real Bott manifolds.
 And the following  Proposition \ref{nilbottclass}  have been proved. See \cite{Mayu} for details.

\begin{pro}\label{nilbottclass}
The $3$-dimensional $S^1$-fibred nilBott manifold of finite type
are
those of $\cG_1$, $\cG_2$, $\cB_1$, $\cB_2$,
$\cB_3$, $\cB_4$.
\end{pro}

\subsection {$3$-dimensional $S^1$-fibred nilBott manifolds of infinite type.}
Any $3$-dimensional $S^1$-fibred nilBott manifold $M_3$ of infinite type is an infranil-Heisenberg manifold.
 The $3$-dimensional simply connected nilpotent Lie group
$N_3$ is isomorphic to the Heisenberg Lie group $\mathsf N$ which is
the product $\mathsf R\times \CC$ with group law:
\[
(x,z)\cdot(y,w)=(x+y-{\rm Im}\bar zw,z+w).\] Then the maximal
compact Lie subgroup of ${\rm Aut}(\mathsf N)$ is ${\rm U}(1)\rtimes
\langle\tau\rangle$ which acts on $\mathsf N$
\begin{equation}\label{auto}
\begin{split}
e^{{\mathbf i}\theta}(x,z)&=(x,e^{{\mathbf i}\theta}z) ,
 \ (e^{{\mathbf i}\theta}\in {\rm U}(1)).\\
\tau(x,z)&=(-x,\bar z).\\
\end{split}\end{equation}
A $3$-dimensional compact infranilmanifold is obtained as a quotient
$\mathsf N/\Gamma$ where $\Gamma$ is a torsionfree discrete uniform
subgroup of ${\rm E}(\mathsf N)=\mathsf N\rtimes ({\rm U}(1)\rtimes
\langle\tau\rangle)$. (See \cite{DE}.)\\
 \ Let
\[ S^1\ra M_3\ra M_2\] be an $S^1$-fibred nilBott manifold of infinite type
which has a group extension $\displaystyle 1\ra \ZZ\ra \pi_3\ra
\pi_2\ra 1$.
 Since
$\mathsf R\subset\mathsf N$ is the center, there is a
commutative diagram of central extensions (\cf
\eqref{inductionnil}):
\begin{equation}\label{inductionnilHeisen}
\begin{CD}
1@>>> \ZZ@>>> \til \Delta_{3}@>>>\Delta_{2}@>>> 1\\
@. @VVV @VVV @VVV @.\\
1@>>> \mathsf R@>>> \mathsf N@>>>\CC@>>> 1.\\
\end{CD}
\end{equation} Using this, we obtain an embedding:
\begin{equation}\label{infinitenilHei}
\begin{CD}
1@>>> \ZZ@>>> \pi_{3}@>>>\pi_{2}@>>> 1\\
@. @VVV @V{\rho}VV @V{\hat\rho}VV @.\\
1@>>> \mathsf R@>>> {\rm E}(\mathsf N)@>>>\CC\rtimes ({\rm
U}(1)\rtimes \langle\tau\rangle)@>>> 1.\\
\end{CD}\end{equation}
Note that $\CC\rtimes ({\rm U}(1)\rtimes
\langle\tau\rangle)=\RR^2\rtimes {\rm O}(2)={\rm E}(2)$.Since $\RR\cap \pi_3=\ZZ$ from \eqref{infinitenilHei}, 
$\hat\rho(\pi_2)$ is a Bieberbach group in ${\rm E}(2)$ so that
$\RR^2/\hat\rho(\pi_2)$ is either $T^2$ or $K$.\\

\noindent{\bf Case (i).}\ Suppose that the holonomy group of $\pi_3$
is trivial. Since $L(\pi_3)=\{1\}$ in ${\rm U}(1)\rtimes \langle\tau\rangle$, it
is noted that $\pi_3=\til \Delta_3$ from \eqref{infinitenilHei} and
\eqref{inductionnilHeisen}. As $\til \Delta_3\leq \mathsf N$, $\til \Delta_3$
is isomorphic to $\Delta(k)$ defined below.\\

 Let $k\in \ZZ$ and define $\Delta(k)$ to be a subgroup of
$\mathsf N$ generated by
\[ c=(2k,0), a=(0,k),b=(0,k{\mathbf i}).\] Put $\mathsf Z=\langle c\rangle$
which is a central subgroup of $\Delta(k)$. It is easy to see that
\begin{equation}\label{bracket} 
 [a,b]=c^{-k}.
\end{equation} Since $\mathsf R$ is the center of
$\mathsf N$, we have a principal bundle
\[
S^1=\mathsf R/\mathsf Z\ra \mathsf N/\Delta(k)\lra \CC/\ZZ^2.\] Then
the euler number of the fibration is $\pm k$.
(See \cite{MIL} for example.)\\

\noindent{\bf Case (ii).}\ Suppose that the holonomy group of
$\pi_3$ is nontrivial. Then we note that $L(\pi_3)=\ZZ_2\leq {\rm
U}(1)\rtimes \langle\tau\rangle$, but not in ${\rm U}(1)$. By \eqref{inductionnil} $L(\pi_3)=L(\pi_2)$, 
so first we
note that $L(\pi_2)$ is not contained in ${\rm U}(1)$.
 Suppose that $(b,A)$
is a element of $\pi_2\leq \RR^2\rtimes {\rm O}(2)$. Then for any $x\in \RR^2$, $(b,A)x\ne x$, because the action of $\pi_2$ on $\RR^2$ is free. Therefore determinant of $(A-E)$ is zero. This implies that if $A\in {\rm SO}(2)$, then $A=E$. So  $L(\pi_2)=L(\pi_3)$ 
not in ${\rm U}(1)$. 
Suppose that there exists an element $g\in \pi_3$ such that
$L(g)=(e^{{\mathbf i}\theta},\tau)\in {\rm U}(1)\rtimes
\langle\tau\rangle$.
 Noting \eqref{auto}, it follows $L(g)^2=1$.
Then $L(\pi_3)=({\rm U}(1)\cap L(\pi_3))\cdot \langle L(g)\rangle$.
Let $\pi_3'=L^{-1}({\rm U}(1)\cap L(\pi_3))\leq \pi_3$ which has the
commutative diagram:
\begin{equation}\label{subgrouptau}
\begin{CD}
1@>>> \ZZ@>>> \pi_{3}@>p_3>>\pi_{2}@>>> 1\\
@. ||@. @AAA @AAA @.\\
1@>>> \ZZ@>>> \pi_3'@>>>\pi_2'@>>> 1.\\
\end{CD}\end{equation}
Since $\pi_2'$ also acts on $\RR^2$  free, 
$L(\pi_3')={\rm U}(1)\cap L(\pi_3)=\{1\}$. 
Hence $L(\pi_3)=\ZZ_2=\langle L(g)\rangle$. 
Note that $M_2$ is the Klein bottle $K$, since $L(\pi_2)=\ZZ_2$. 
Let $n=(x,0)$ be a generator of $\ZZ\leq \mathsf N$. 
Choose $h\in \pi_3$ with $L(h)=1$ such that the subgroup 
$\langle p_3(g),p_3(h)\rangle $ is the fundamental group of $K$. 
It has a relation
$p_3(g)p_3(h)p_3(g)^{-1}=p_3(h)^{-1}$. Then $\langle n,g,h\rangle$ is
isomorphic to $\pi_3$. In particular, those generators satisfy
\begin{equation}\label{relation1}
\begin{split}
ghg^{-1}&=n^k h^{-1} \, ({}^\exists\, k\in\ZZ),\\
gng^{-1}&=L(g)n=\tau n=n^{-1}, \ hnh^{-1}=L(h)n=n.
\end{split}\end{equation}

 On the other hand, let $\Gamma(k)$ be a subgroup of ${\rm E}(\mathsf
N)$ generated by
\begin{equation}\label{gam} n=\left((k,0),I\right),\, \al=\left((0,\frac k2),\tau\right),\,
\be=\left((0,k{\mathbf i}),I\right).
\end{equation}Note that
$\al^2=\left((0,k),I\right)$. Then it is easily checked that
\begin{equation}\begin{split}
\al n\al^{-1}=n^{-1},\, \al\be\al^{-1}=n^k\be^{-1},\, \be
n\be^{-1}=n.
\end{split}\end{equation}

\begin{equation}\label{infinitenilHei1}
\begin{CD}
1@>>> \mathsf R@>>> {\rm E}(\mathsf N)@>>>\CC\rtimes ({\rm
U}(1)\rtimes \langle\tau\rangle)@>>> 1\\
@. @AAA @AAA @AAA @.\\
1@>>> \langle n\rangle @>>> \Gamma(k)@>>> \langle \hat \al,\hat\be\rangle@>>> 1.\\
 \end{CD}
\end{equation} Then the subgroup generated by
$\hat\al^2,\hat\be$ is isomorphic to the subgroup of
translations of $\RR^2$; $\displaystyle t_1=\left(\begin{array}{c} k\\
0\end{array}\right), t_2=\left(\begin{array}{c} 0\\
k\end{array}\right)$. Let $T^2=\RR^2/\langle t_1,t_2\rangle$. Then
it is easy to see that the element $\ga=[\hat\al]$ of order $2$ acts
on $T^2$ as
\begin{equation}\label{quoC}
\ga(z_1,z_2)=(-z_1,\bar z_2).
\end{equation}As a consequence, $\RR^2/\langle \hat\al,\hat\be\rangle=
T^2/\langle\ga\rangle$ turns out to be $K$. So $M_3=\mathsf
N/\Gamma(k)$ is an $S^1$-fibred nilBott manifold:
\[ S^1\ra \mathsf N/\Gamma(k)\ra K\]where
$S^1=\mathsf R/\langle n\rangle$ is the fiber (but not an action).
\\

Compared \eqref{relation1} with $\Gamma(k)$, $\pi_3$ is isomorphic
to $\Gamma(k)$ with the following commutative arrows of
isomorphisms:
\begin{equation}\label{comparison}
\begin{CD}
1@>>> \ZZ@>>> \pi_{3}@>>>\pi_{2}@>>> 1\\
@. @VVV @VVV @VVV @.\\
1@>>> \langle n\rangle@>>> \Gamma(k)@>>>\langle \hat \al,\hat\be\rangle@>>> 1.\\
\end{CD}\end{equation}
As both $(\pi_3,X_3)$ and $(\Gamma(k),\mathsf N)$ are Seifert
actions, the isomorphism of \eqref{comparison} implies that they are
equivariantly diffeomorphic, \ie $M_3=X_3/\pi_3\cong \mathsf
N/\Gamma(k)$.\\
This shows the following. 
\begin{pro}\label{4-nils1}
 A $3$-dimensional an $S^1$-fibred nilBott manifold $M_3$ of infinite type
is either a Heisenberg nilmanifold $\mathsf N/\Delta(k)$ or an
infranilmanifold $\mathsf N/\Gamma(k)$.
\end{pro}

\section{Realization}\label{real}
Let $Q=\pi_1(K)$ be the fundamental group of $K$. $Q$ has a presentation:
\begin{equation}\label{eq:geneG}
\{ g,h\,|\, ghg^{-1}=h^{-1}\}.
\end{equation}
A group extension $1\ra \ZZ\ra \pi\ra Q\ra 1$ for any $3$-dimensional $S^1$-fibred nilBott manifold over
$K$ represents a $2$-cocycle in $H^2_\phi(Q,\ZZ)$ for some
representation $\phi$. Conversely, given a representation $\phi$, we
may show any element of $H^2_\phi(Q,\ZZ)$ can be realized
 as an  $S^1$-fibred nilBott manifold. \\

We must consider following cases of a representation $\phi$:
\begin{equation*}
\begin{split}
&{\bf Case \, 1.} \ \phi(g)=1, \ \phi(h)=1,\\ 
&{\bf Case \, 2.} \ \phi(g)=1, \ \phi(h)=-1, \\
&{\bf Case \, 3.} \ \phi(g)=-1, \ \phi(h)=1,\\ 
&{\bf Case \, 4.} \ \phi(g)=-1, \ \phi(h)=-1.  
\end{split}
\end{equation*}
Let  $\phi_i$ $(i=1,2,3,4)$ be  the representation $\phi$ of the {\bf Casei} above.
 A $2$-cocycle $[f_k]\in H^2_{\phi_i}(Q,\ZZ)$ gives rise
to a group extension
\[ 1\ra \ZZ\ra {}_i\pi(k)\stackrel{p}\lra G\ra 1,\]
where ${}_i\pi(k)$ is generated by $\tilde g$, $\tilde h$, $n$ such that $\langle n\rangle=\ZZ$, $p(\tilde g)=g$,
 $p(\tilde h)=h$. By \eqref{eq:geneG}, 
\begin{equation}\label{eq:genepi(k)}
\tilde g\tilde h\tilde g^{-1}=n^k\tilde h^{-1}.
\end{equation}
for some $k\in \ZZ$. Note that $[f_0]=0$.\\ 

{\bf Case\,1:} Since $\phi_1$ is trivial, $H^2_{\phi_1}(Q,\ZZ)=H^2(Q,\ZZ)\approx H^2(K,\ZZ)\approx \ZZ_2.$ 
Moreover 
${}_1\pi(k)$ satisfies the presentation:
\begin{equation}\label{eq:onepi(k)}
\tilde gn\tilde g^{-1}=n, \, \tilde hn\tilde
h^{-1}=n, \, 
 \tilde g\tilde h\tilde g^{-1}=n^k\tilde h^{-1}.
\end{equation}

\begin{lmm}\label{lem:split}
The groups ${}_1\pi(0)$, ${}_1\pi(1)$ are isomorphic to $\pi_1(\mathcal B_1)$, $\pi_1(\mathcal B_2)$ respectivery.
\end{lmm}

\begin{proof}
First we discuss about ${}_1\pi(0)$. Let $\til g, \til h, n\in {}_1\pi(0)$ be as above. Put $\varepsilon=\tilde g$, \, $t_1=\tilde g^2$, \, $t_2=n$,\, and
$t_3=\tilde h$.  Note that the group
which generated by $\varepsilon,t_1,t_2,t_3$ coincides with ${}_1\pi(0)$. 
Using the relation \eqref{eq:onepi(k)}, 
\begin{equation*}
\begin{split}
\varepsilon^2&=t_1, \\
\varepsilon t_2\varepsilon^{-1}&=\til g\til h\til g^{-1}=\til {h}^{-1}=t_2^{-1}, \\
\varepsilon t_3\varepsilon^{-1}&=\til gn\til g^{-1}=n=t_3.
\end{split}
\end{equation*}Compared this relation with $\pi_1(\cB_1)$, ${}_1\pi(0)$ is isomorphic to $\pi_1(\cB_1)$. $($in the Wolf's
notation \cite{WO}$)$\\
 \ Second, about ${}_1\pi(1)$. Let $\til g, \til h, n\in {}_1\pi(1)$ be as above. Put $\varepsilon=\tilde g$, $t_1=\tilde g^2$,
$t_2=\tilde g^{-2} n$, and $t_3=h$. The group
which generated by $\varepsilon,t_1,t_2,t_3$ coincides with ${}_1\pi(1)$. By using the relation \eqref{eq:onepi(k)},
\begin{equation*}
\begin{split}
\varepsilon^2&=t_1, \\
\varepsilon t_2\varepsilon^{-1}&=\til g\til g^{-2}n\til g^{-1}=\til {g}^{-1}n\til g^{-1}=\til {g}^{-2}n=t_1, \\
\varepsilon t_3\varepsilon^{-1}&=\til gh\til g^{-1}=\til {g}^{2}\til {g}^{-2}n\til {h}^{-1}=t_1t_2t_3^{-1}. \\
\end{split}
\end{equation*}This implies that
${}_1\pi(1)$ is isomorphic to $\pi_1(\mathcal B_2)$.  $($See \cite{WO}$)$
\end{proof}

Remark that the fundamental group $\pi_1(\mathcal B_2)$ is isomorphic to ${}_1\pi(1)$ so we have a group extension which represents the $[f_1]\in H^2_{\phi_1}(Q,\ZZ)$ with $2$-torsion. Therefore, if $[f_k]\ne 0$, then $k=1.$\\

{\bf Case\,2:} 
Let $\phi_2(g)=1, \phi_2(h)=-1$, then 
${}_2\pi(k)$ has the following
presentation. 
\begin{equation}\label{eq:twopi(k)}
\tilde gn\tilde g^{-1}=n, \ \tilde hn\tilde
h^{-1}=n^{-1}, \ \tilde g\tilde h\tilde g^{-1}=n^k\tilde h^{-1}.
\end{equation}
for some $k\in \ZZ$,
\begin{pro}\label{lem:split}
The groups ${}_2\pi(0)$, ${}_2\pi(1)$ are isomorphic to $\pi_1(\mathcal B_3)$, $\pi_1(\mathcal B_4)$ respectivery.
\end{pro}

\begin{proof}
Let $\til g, \til h, n\in {}_2\pi(0)$ be as before. put $\al=\tilde h\tilde g$, $\varepsilon=\tilde h^{-1}$, $t_1=\tilde g^2$,
$t_2=\tilde h^{-2}$, and $t_3=n$. Note that the group
generated by $\al,\varepsilon,t_1,t_2,t_3$ coincides with ${}_2\pi(0)$.
 Using the
relation \eqref{eq:twopi(k)}, 
\begin{equation*}
\begin{split}
\tilde \al^2&=(\tilde h\tilde g)^2=\tilde h\tilde h^{-1}\tilde g\tilde g=\til g^2=t_1,\\
\varepsilon^2&=t_2,\\
\varepsilon\al\varepsilon^{-1}&=\tilde h^{-1}\tilde h\tilde g\tilde h=\tilde h^{-1}\tilde g=t_2\al.\\
\al t_2\al^{-1}&=\tilde h\tilde g\tilde h^{-2}\tilde g^{-1}\tilde h^{-1}=\tilde h^{-2}=t_2^{-1},\\
al t_3\al^{-1}&=\tilde h\tilde gn\tilde g^{-1}\tilde h^{-1}=n^{-1}=t_3^{-1},\\
\varepsilon t_1\varepsilon^{-1}&=\til h^{-1}\til g^2\til h=\tilde h^{-1}\tilde g\tilde h^{-1}\tilde g=\tilde h^{-1}\tilde h\tilde g\tilde g=\til g^2=t_1, \\
\varepsilon t_3\varepsilon^{-1}&=\til h^{-1}n\til h=n^{-1}=t_3^{-1}.\\
\end{split}
\end{equation*}
This relation correspond to of $\pi_1(\mathcal
B_3)$.  $($See \cite{WO}$)$.
So ${}_2\pi(0)$ is isomorphic to $\pi_1(\mathcal
B_3)$.\\
Let $\til g, \til h, n\in {}_2\pi(1)$ be as above. Put $\al=\tilde h\tilde g$, $\varepsilon=n^{-1}\tilde h^{-1}$, $t_1=n^{-1}\tilde g^{2}$, $t_2=\tilde h^{-2}$, 
 and $t_3=n^{-1}$.Using the
relation \eqref{eq:twopi(k)}, we obtain a presentation:
\begin{equation*}
\begin{split}
\tilde \al^2&=(\tilde h\tilde g)^2=\tilde hn\tilde h^{-1}\tilde g\tilde g=n^{-1}\til g^2=t_1,\\
\varepsilon^2&=t_2,\\
\varepsilon\al\varepsilon^{-1}&=n^{-1}\tilde h^{-1}\tilde h\tilde g\tilde hn=\tilde h^{-1}\tilde gn=t_2t_3\al.\\
\al t_2\al^{-1}&=\tilde h\tilde g\tilde h^{-2}\tilde g^{-1}\tilde h^{-1}=\tilde h^{-2}=t_2^{-1},\\
\al t_3\al^{-1}&=\tilde h\tilde gn^{-1}\tilde g^{-1}\tilde h^{-1}=n=t_3^{-1},\\
\varepsilon t_1\varepsilon^{-1}&=n^{-1}\til h^{-1}n^{-1}\til g^2\til hn=n^{-1}{\til t_2}^2=t_1, \\
\varepsilon t_3\varepsilon^{-1}&=n^{-1}\til h^{-1}n^{-1}\til hn=n=t_3^{-1}.\\
\end{split}
\end{equation*}
This implies that
${}_2\pi(1)$ is isomorphic to $\mathcal B_4$.  $($See  \cite{WO}$)$ 
\end{proof}

\begin{pro}\label{two}
Any element of $H^2_{\phi_2}(Q,\ZZ)$ is isomorphic to $\ZZ_2$
\end{pro}

\begin{proof}
Let $Q'$ be the subgroup of $Q$ which is generated by $g,h^2\in Q$ with 
$gh^2g^{-1}=(ghg^{-1})^2=h^{-2}$. 
We have a commutative diagram:
\begin{equation}\label{eq:induce}
\begin{CD}
1@>>> \ZZ@>>> {}_2\pi(k)@>p>>Q@>>> 1\\
@. ||@. @AAA @AAA @.\\
1@>>> \ZZ@>>> \pi'@>p>>Q'@>>> 1\\
\end{CD}\end{equation}where
$\displaystyle \pi'$ is the subgroup of ${}_2\pi(k)$ which is generated by $ n, \tilde g, \tilde h^2$. Note
that 
\[
\tilde g\tilde h^2\tilde g^{-1}=n^k\tilde h^{-1}n^k\tilde h^{-1}=
\tilde h^{-2}.\] Since the subgroup $\langle \tilde g, \tilde
h^2\rangle$ of $\displaystyle \pi'$ maps isomorphically onto $Q'$ and a restriction 
$\phi\big|Q'={\rm id}$, then $\pi'=\ZZ\times Q'$. This shows
that the restriction homomorphism
$\displaystyle\iota^*:H^2_{\phi_2}(Q,\ZZ)\ra H^2(Q'.\ZZ)$ is the zero
map, equivalently $\iota^*[f_k]=0$. Using the transfer homomorphism
$\tau: H^2(Q',\ZZ)\ra H^2_{\phi_2}(Q,\ZZ)$ and by the property
$\tau\circ\iota^*([f])=[Q:Q'][f]=2[f]$ $(\forall\, [f]\in H^2_{\phi_2}(Q,\ZZ)$),
we obtain $2[f]=0$.\\
 \ On the other hand, from \eqref{eq:twopi(k)}  
\begin{equation}\label{fk}
\begin{split}n^k=\tilde g\tilde h\tilde g^{-1}\tilde h&=(0,g)(0,h)(-f_{k}(g^{-1},g),g^{-1})(0,h)\\
                             &=f_k(g,h)+f_k(g^{-1},g)+f_k(gh,g)+f_k(h^{-1},h).
\end{split}
\end{equation}
Since there exists a $2$-cocycle $[f_1]$ from Proposition \ref{lem:split},  
\begin{equation*}
n=f_1(g,h)+f_1(g^{-1},g)+f_1(gh,g)+f_1(h^{-1},h),
\end{equation*} and
\begin{equation}\label{KK}
n^k=kf_1(g,h)+kf_1(g^{-1},g)+kf_1(gh,g)+kf_1(h^{-1},h).
\end{equation}
Noting that $[f_k]$ represents of ${}_2\pi(k)$ if and only if $f_k$ satisfies \eqref{fk}, the relation \eqref{KK} shows that
\begin{equation}\label{eq:k}
[f_k]=k\cdot [f_1]. \end{equation}                           
As the consequence, $H^2_{\phi_2}(Q,\ZZ)$ is isomorphic to $\ZZ_2$. 
\end{proof}
This gives the following result:
\begin{crl} 
The group extension ${}_2\pi(k)$ is isomorphic to $\pi_1(\mathcal B_3)$ or $\pi_1(\mathcal B_4)$ accordance with $k$ is odd or even.
\end{crl}

 {\bf Case\,3:}
The group ${}_3\pi(k)$ has the following
presentation. For some $k\in \ZZ$,
\begin{equation}\label{eq:threepi}
\tilde gn\tilde g^{-1}=n^{-1}, \ \tilde hn\tilde
h^{-1}=n, \ 
\tilde g\tilde h\tilde g^{-1}=n^k\tilde h^{-1}.
\end{equation}

\begin{lmm}\label{lem:split1}
The groups ${}_3\pi(0)$, ${}_3\pi(k)$ are isomorphic to $\pi_1(\mathcal G_2)$, $\Gamma (k)$ respectivery. (\cf \eqref{gam})
\end{lmm}

\begin{proof}
Let $\til g, \til h, n\in {}_3\pi(0)$ be as before. 
Put $\al=\tilde g$, $t_1=\tilde g^2$,
$t_2=\tilde h$, and $t_3=n$. Note that the group
generated by $\al,t_1,t_2,t_3$ coincides with ${}_3\pi(0)$.
 By using the
relation \eqref{eq:threepi}, it is easy to check that:
\begin{equation*}\begin{split}
&\al^2=t_1,\\
&\al t_2\al^{-1}=t_2^{-1},\\
&\al t_3\al^{-1}=t_3^{-1}.\\
\end{split}
\end{equation*}
So ${}_3\pi(0)$ is isomorphic to $\mathcal
G_2$. (See \cite{WO}.)\\
Suppose $\til g, \til h,n\in {}_3\pi(k)$.
Put $\al=\tilde g$, $\be=\tilde h$. This implies that
${}_3\pi(k)$ is isomorphic to $\Gamma(k)$.
\end{proof}

\begin{pro}\label{tor}
$H^2_{\phi_3}(G,\ZZ)$ is isomorphic to $\ZZ$.
\end{pro}
\begin{proof}
 From Theorem \ref{s1-case} and Lemma \ref{lem:split3}, $H^2_{\phi_3}(G,\ZZ)$ is torsionfree. Moreover, it is satisfies \eqref{eq:k}. Therefore $H^2_{\phi_3}(G,\ZZ)$ is isomorphic to $\ZZ$.
\end{proof}

{\bf Case\,4.}   
The group ${}_4\pi(k)$ has the following
presentation. 
\begin{equation}\label{eq:fourpi}
\tilde gn\tilde g^{-1}=n^{-1}, \tilde hn\tilde
h^{-1}=n^{-1},
 \tilde g\tilde h\tilde g^{-1}=n^k\tilde h^{-1}.
\end{equation}
Put $\al=gh$. It is easy to check that 
\begin{equation}
\al n\al^{-1}=n, \ \til hn\til h^{-1}=n^{-1}, \ 
\al \til h\al=n^k\til h^{-1}
\end{equation}
Noting that ${}_4\pi(k)$ coinside with the group which is generated by $\al, \til h$ and $n$, we can show that ${}_4\pi(k)$ is isomorphic to ${}_2\pi(k)$.   

We have shown that any element of $H^2_\phi(Q,\ZZ)$ can be realized an
$S^1$-fibred nilBott manifold, and obtain the following table.
\[
\vbox{
\offinterlineskip
\halign{\strut#&&\vrule#&\quad\hfil#\hfil\quad\cr
\noalign{\hrule}
&& &&  && Case 1 && Case2 and 4 && Case3 &\cr
\noalign{\hrule}
&& && $H^2_\phi(Q,\ZZ)$ && $\ZZ_2$ && $\ZZ_2$ && $\ZZ$ &\cr
\noalign{\hrule}
&&  && $[f]=0$ && $\pi_1(\mathcal B_1)$ && $\pi_1(\mathcal B_3)$ && $\pi_1(\mathcal G_2)$ &\cr
\cline{5-12}
&& $\pi_1(M_3)$ && $[f]\ne 0$:torsion && $\pi_1(\mathcal B_2)$ && $\pi_1(\mathcal B_4)$ && - &\cr
\cline{5-12}
&& && $[f]$:torsionfree && - && - && $\Gamma(k)$ &\cr
\noalign{\hrule}
}}
\]

Let $\ZZ^2=\pi_1(T^2)$ be the fundamental group of $T^2$ which is generated by $\al, \be$. 

Given a representation $\phi$, we
may show any element of $H^2_\phi(\ZZ^2,\ZZ)$ can be realized
 as an $S^1$-fibred nilBott manifold. \\

We must consider following cases of a representation $\phi$:
\begin{equation*}
\begin{split}
&{\bf Case\,5.} \ \phi(\al)=1, \ \phi(\be)=1,\\ 
&{\bf Case\,6.} \ \phi(\al)=1, \ \phi(\be)=-1, \\ 
&{\bf Case\,7.} \ \phi(\al)=-1, \ \phi(\be)=-1.  
\end{split}
\end{equation*}
Let denote $\phi_i$ as before. In each case, 
a $2$-cocycle $[f_k]\in H^2_{\phi_i}(\ZZ^2,\ZZ)$ gives rise
to a group extension
\[ 1\ra \ZZ\ra {}_i\pi(k)\stackrel{p}\lra \ZZ^2\ra 1,\]
where ${}_i\pi(k)$ is generated by $\tilde \al$, $\tilde \be$, $m$ such that $\langle m\rangle=\ZZ$, $p(\tilde \al)=\al$,
 $p(\tilde \be)=\be$. By \eqref{eq:geneG}, 
\begin{equation}\label{eq:genepi(k)}
\tilde \al\tilde \be\tilde \al^{-1}=m^k\tilde \be.
\end{equation}
for some $k\in \ZZ$. 

{\bf Case5:}  
The group ${}_5\pi(k)$ has the following
presentation. 
\begin{equation}\label{eq:5twopi(k)}
\tilde \al m\tilde \al^{-1}=m, \ \tilde \be m\tilde
\be^{-1}=m, \ \tilde \al\tilde \be\tilde \al^{-1}=m^k\tilde \be.
\end{equation}
for some $k\in \ZZ$, this relation gives the following proposotion. (See \eqref {bracket}).
\begin{pro}\label{lem:split2}
The groups ${}_5\pi(0)$, ${}_5\pi(k)$ are isomorphic to $\pi_1(T^3)$, $\pi_1(\Delta(-k))$ respectivery.
\end{pro}

{\bf Case\,6:} 
The group ${}_6\pi(k)$ has the following
presentation. 
\begin{equation}\label{eq:6twopi(k)}
\tilde \al m\tilde \al^{-1}=m, \ \tilde \be m\tilde
\be^{-1}=m^{-1}, \ \tilde \al\tilde \be\tilde \al^{-1}=m^k\tilde \be.
\end{equation}
for some $k\in \ZZ$. 
\begin{pro}\label{lem:split3}
The groups ${}_6\pi(0)$, ${}_6\pi(1)$ are isomorphic to $\pi_1(\mathcal B_1)$, $\pi_1(\mathcal B_2)$ respectivery.
\end{pro}
\begin{proof}
First let $k=0$. Put $m=\til h$, $\til \al=n$, $\til \be=\til g$, then we can check easilly that 
 ${}_6\pi(0)$ is isomorphic to ${}_1\pi(0)$, \ie  $\tilde gn\tilde g^{-1}=n$, \, $\tilde hn\tilde
h^{-1}=n$, \,  $\tilde g\tilde h\tilde g^{-1}=\tilde h^{-1}$. 
So ${}_6\pi(0)$ is isomorphic to $\pi_1(\mathcal B_1)$.\\
 Second suppose $k=1$. Put $m=n$, $\til \al=\til g$, $m^{-1}\til \be=\til h$, then we can check easilly that 
 ${}_6\pi(1)$ is isomorphic to $\pi_1(\mathcal B_2)$ in the same way above.
\end{proof}
Moreover we can obtain the following after the fashion of proof for the Proposition \ref{two}

{\bf Case\,7:} 
The group ${}_7\pi(k)$ has the following
presentation. 
\begin{equation}\label{eq:7twopi(k)}
\tilde gn\tilde g^{-1}=n^{-1}, \ \tilde hn\tilde
h^{-1}=n^{-1}, \ \tilde g\tilde h\tilde g^{-1}=n^k\tilde h
\end{equation}
for some $k\in \ZZ$. Then it easy check that ${}_7\pi(k)$ is isomorphic to ${}_6\pi(k)$ in the same way for Case4 above.

As a consequence, we obtain a table:
\[
\vbox{
\offinterlineskip
\halign{\strut#&&\vrule#&\quad\hfil#\hfil\quad\cr
\hline
&& &&  && Case 1 && Case2 and 3  &\cr
\noalign{\hrule}
&& && $H^2_\phi(\ZZ^2,\ZZ)$ && $\ZZ$ && $\ZZ_2$ &\cr
\noalign{\hrule}
&&  && $[f]=0$ && $\ZZ^3$ && $\pi_1(\mathcal B_1)$  &\cr
\cline{5-10}
&& $\pi_1(M_3)$ && $[f]\ne 0$:torsion && - && $\pi_1(\mathcal B_2)$  &\cr
\cline{5-10}
&& && $[f]$:torsionfree && $\Delta(k)$ && -  &\cr
\noalign{\hrule}
}}
\]

\begin{thm}[Halperin-Carlsson conjecture \cite{PU}]\label{CA}
Let $T^s$ be an arbitrary effective action on an $m$-dimensional
$S^1$-fibred nilBott manifold $M$ of finite type.
Then
\begin{equation}
{}_sC_j\leq b_j\ (=\mbox{the j-th Betti number of}\,  M).
\end{equation}
In particular $\displaystyle 2^s\leq \mathop{\sum}_{j=0}^m {\rm
Rank}\, H_j(M)$.
\end{thm}

\begin{proof}By Corollary \ref{torus-action},
$M$ admits a homologically injective $T^k$-action where
$k={\rm Rank}\, C(\pi)$ where $\pi=\pi_1(M)$.
Then we have shown in \cite{KMa} that any
homologically injective $T^k$-actions on any closed aspherical
manifold satisfies that
\begin{equation*}
{}_kC_j\leq b_j.
\end{equation*}
It follows from the result of Conner-Raymond\cite{C-R} that
there is an injective homomorphism $\displaystyle
1\ra \ZZ^s\ra C(\pi)$. This shows that $s\leq k$ so we obtain
\begin{equation}
{}_sC_j\leq b_j\ (=\mbox{the j-th Betti number of}\,  M).
\end{equation}

\end{proof} 

\begin{rem}
This result is obtained when $M_i$ is a real Bott manifold by
Masuda, Choi and Oum. 
\end{rem}

\end{document}